\documentclass[11pt, reqno]{amsart}

\usepackage[hmargin=3.9cm,vmargin=4cm]{geometry}

\usepackage{graphicx}
\usepackage[hyphens,spaces,obeyspaces]{url}
\usepackage[colorlinks,allcolors=blue]{hyperref}

\usepackage{amscd,amssymb}
\usepackage[leqno]{amsmath}
\usepackage{tikz}
\usepackage{tikz-3dplot}
\usepackage{tikz-cd}
\usepackage{faktor}
\usepackage{amsthm,amssymb,mathtools}

\usepackage[utf8]{inputenc}
\usepackage{hyperref}
\usepackage{cleveref}

\newtheorem  {theorem}                  {Theorem}
\newtheorem* {theorem*}                   {Theorem}

\newtheorem {lemma}[theorem] {Lemma}
\newtheorem {prop}[theorem]      {Proposition}
\crefname{prop}{proposition}{propositions}
\newtheorem* {prop*}     {Proposition}

\theoremstyle{definition}
\newtheorem {defi}[theorem] {Definition}
\newtheorem {Remark}[theorem]        {Remark}
\newtheorem* {Remark*} {Remark}

\newtheorem* {Example*}    {Example}

\def\R{\mathbb{R}}
\def\Z{\mathbb{Z}}

\newcommand{\pp}[2]{\frac{\partial#1}{\partial#2}}
\newcommand{\norm}[1]{\|#1\|}

\newcommand{\cF}{\mathcal{F}}

\DeclareMathOperator{\const}{const}
\DeclareMathOperator{\Id}{Id}
\newcommand{\shs}{stable Hamiltonian structure }
\renewcommand{\d}{\mathrm{d}}

\title[Non-density results in stable Hamiltonian topology]{Non-density results in high dimensional \\ stable Hamiltonian topology}
\author{Robert Cardona}
\address{Robert Cardona, Departament de Matem\`atiques i Inform\`atica, Universitat de Barcelona, Gran Via de Les Corts Catalanes 585, 08007 Barcelona, Spain. Centre de Recerca Matemàtica, Campus de Bellaterra, Edifici C, 08193, Barcelona, Spain. \it{e-mail: robert.cardona@ub.edu}}
\thanks{RC acknowledges partial support from the AEI grant PID2023-147585NA-I00, the Departament de Recerca i Universitats de la Generalitat de Catalunya (2021 SGR 00697), and the Spanish State Research Agency, through the Severo Ochoa and María de Maeztu Program for Centers and Units of Excellence in R\&D (CEX2020-001084-M)}

\author{Fabio Gironella}
\address{Fabio Gironella, CNRS - Laboratoire de Mathématiques Jean Leray, Nantes Université, France. \it{E-mail: fabio.gironella@cnrs.fr}}
\thanks{FG benefits from the support of the ANR Grant COSY ``New challenges in symplectic and contact topology'', bearing the reference ANR-21-CE40-0002; of the region Pays de la Loire, via the project Étoile Montante 2023 SymFol; and of the French government's “Investissements d’Avenir” program integrated to France 2030, bearing the following reference ANR-11-LABX-0020-01.}

\date{}

\begin{document}

\maketitle

\begin{abstract}
   We push forward the study of higher dimensional stable Hamiltonian topology by establishing two non-density results. 
   First, we prove that stable hypersurfaces are not $C^3$-dense in any isotopy class of embedded hypersurfaces on any ambient symplectic manifold of dimension $2n\geq 8$. 
   Our second result is that on any manifold of dimension $2m+1\geq 5$, the set of non-degenerate stable Hamiltonian structures is not $C^2$-dense among stable Hamiltonian structures in any given stable homotopy class that satisfies a mild assumption. 
   The latter generalizes a result by Cieliebak and Volkov to arbitrary dimensions.
\end{abstract}

\section{Introduction}

A \emph{stable Hamiltonian structure} on a manifold $M$ of odd dimension $2n+1$ is a pair $(\lambda,\omega)$, where $\lambda$ is a one-form and $\omega$ is a closed two-form of maximal rank, satisfying $\lambda\wedge \omega^n > 0$ and $\ker \omega \subset \ker d\lambda$.
These structures can be understood as a generalization of contact forms, since a contact form $\alpha$ on $M$ naturally gives a stable Hamiltonian structure $(\alpha,d\alpha)$. 
Like contact manifolds, stable Hamiltonian structures arise in certain hypersurfaces of symplectic manifolds, named \emph{stable hypersurfaces}. 
These were originally defined by Hofer and Zehnder \cite{HZ} as hypersurfaces that admit a distinguished symplectic tubular neighborhood $U\cong M\times (-\varepsilon,\varepsilon)$ such that the characteristic foliation on each slice $M\times \{t\}$ is conjugate to that of $M$. In analogy with contact topology, \emph{stable Hamiltonian topology} refers to the investigation of the topological properties of stable Hamiltonian structures. 
This term was introduced by Cieliebak and Volkov in their foundational work \cite{CV}, and was defined as the study of the set of stable Hamiltonian structures in $M$ up to homotopy with a fixed cohomology class of $\omega$. 
Each of these homotopy classes is called a \emph{stable homotopy class}. 
We interpret stable Hamiltonian topology in a slightly broader sense; for instance, one can also study the set of embedded stable hypersurfaces on a given symplectic manifold $(W,\Omega)$, and how it lies within the set of all embedded hypersurfaces in $W$. \\

The goal of this article is to push forward the investigation of higher dimensional stable Hamiltonian topology. When the dimension of $M$ is three, fundamental results were established in \cite{CV} and other works investigated the topology \cite{CV2, C} and the dynamics \cite{HT, R, CR} of stable Hamiltonian structures. Stability of hypersurfaces is a recurrent topic in symplectic field theory \cite{BEHWZ, EGH} and more generally in symplectic topology \cite{EKP, CFP2, MP, NW}. In higher dimensions, a few results can be found in \cite{CV} and it was shown in \cite{CFP} that there are examples of six-dimensional symplectic manifolds where stability is not an open condition (see \cite{C} for a general result in four dimensions). 

On the other hand, the question of the density (or non-density) of stable hypersurfaces has not been addressed yet in full generality. In dimension six, the particular symplectic manifolds with boundary constructed by Cieliebak-Frauenfelder-Paternain \cite{CFP} admit open sets (in $C^k$ topology, with $k\geq 2$) of non-stable hypersurfaces. 
As we will see in Section \ref{ss:An}, there are also examples of four-dimensional symplectic manifolds (which are not closed either) where one can find such open sets. 
However, these are just concrete examples of symplectic manifolds and isotopy classes of embedded hypersurfaces where this phenomenon is shown to occur. 
None of them covers an ambient closed symplectic manifold or, say, the standard symplectic Euclidean space (where contact hypersurfaces are not dense even in the Haussdorff topology \cite{C95}).
Given a symplectic manifold $(W,\Omega)$, we denote by $\mathcal{SHS}(W)$ the set of embedded stable hypersurfaces, and by $\mathcal{HS}(W)$ the set of embedded hypersurfaces. Our first result proves that $\mathcal{SHS}(W)$ is \emph{never} $C^3$-dense in $\mathcal{HS}(W)$ (and hence also not $C^k$-dense for all $k\geq 3$) for an arbitrary symplectic manifold $W$ of dimension at least eight and any isotopy class of embedded hypersurfaces.
\begin{theorem}\label{thm:main1}
    Let $(W,\Omega)$ be a symplectic manifold such that $\dim W\geq 8$. Let $M$ be an embedded hypersurface in $W$. Then $M$ is isotopic to a $C^0$-close hypersurface $\widetilde M$ that cannot be $C^3$-approximated by stable hypersurfaces.
\end{theorem}
The idea of the proof is to introduce in the characteristic foliation of a hypersurface a normally hyperbolic invariant submanifold that plays the role of a robust obstruction to stability. There is a single argument in the proof of Theorem \ref{thm:main1} that requires an adaptation to work in dimension six, see Remark \ref{rem:dim6}. An analogous statement in six and four dimensions remains open to our knowledge. 
The need for at least $C^3$-approximations instead of just $C^2$ comes from a technical issue related to the persistence of normally hyperbolic invariant submanifolds; this is discussed in more detail in Remark \ref{rem:regu}.
\\ 

In the second part of this work, we establish a non-density result for non-degenerate stable Hamiltonian structures in a given manifold of dimension $2n+1$. To this end, we recall that a stable Hamiltonian structure $(\lambda,\omega)$ is called \emph{non-degenerate} if its Reeb vector field $R$, which is defined by the equations $\iota_R\omega=0$ and $\lambda(R)=1$, satisfies the following property: for every closed orbit $\gamma\colon \R\to M$ of $R$  the linearized Poincar\'e map of $R$ along $\gamma$ does not have any root of the unity among its eigenvalues.
In dimension three, it was proven by Cieliebak and Volkov in \cite[Theorem 1.10]{CV} that in every stable homotopy class there is a stable Hamiltonian structure that cannot be $C^2$-approximated by non-degenerate stable Hamiltonian structures. We establish an analogous result for stable homotopy classes satisfying a mild condition in any dimension.

\begin{theorem}
\label{thm:stably_degenerate}
    Let $M$ be a manifold of dimension $2n+1\geq 5$. In each regular stable homotopy class, there is a \shs $(\lambda',\omega')$ that cannot be $C^2$-perturbed to a non-degenerate stable Hamiltonian structure.
\end{theorem}

\noindent
A stable homotopy class is \emph{regular} if there is a pair $(\lambda,\omega)$ in it satisfying that $d\lambda$ is a non-zero constant multiple of $\omega$ in some open set $U\subset M$, see Definition \ref{def:goodSHS}. 
We don't know if every stable homotopy class is regular, but this is the case in dimension three by \cite[Proposition 3.31]{CV}. The most natural examples of stabilizable Hamiltonian structures lie in regular homotopy classes, such as contact type Hamiltonian structures, symplectic mapping tori, or products of contact and symplectic manifolds, see Section \ref{ss:regular}. 

We point out that Theorem \ref{thm:stably_degenerate} and its three-dimensional analog do not imply the existence of stable hypersurfaces that cannot be $C^2$-approximated by non-degenerate stable hypersurfaces, see Remark \ref{rem:approx}.
We also mention that both Theorem \ref{thm:main1} and \ref{thm:stably_degenerate} do not require the manifold $M$ (or $W$) to be closed, even though we are mostly interested in the closed case.

For the proof of \Cref{thm:stably_degenerate}, we introduce a set of ``coupling functions" that are first integrals of the Reeb field of a \shs (see Lemma \ref{lem:first_integrals}) and that generalize the classical function $f=\frac{d\lambda}{\omega}$ only defined in three dimensions. Then, we construct via a stable homotopy a \shs that admits, robustly, a periodic orbit in a regular level set of one of these first integrals. 
This periodic orbit turns out to be always degenerate (see Lemma \ref{lem:degenerate_Hamilt_Reeb}). 
\\

\textbf{Organization of this paper.} In Section \ref{s:pre}, we introduce stable Hamiltonian structures and a set of coupling functions associated with $(\lambda,\omega)$ in arbitrary dimensions. We recall as well some results in dynamical systems that will be used in our proofs. Sections \ref{s:T1} and \ref{s:T2} are independent and develop the proofs of Theorems \ref{thm:main1} and \ref{thm:stably_degenerate} respectively.\\

\textbf{Acknowledgements.} The authors are grateful to Jonathan Bowden for his suggestions to prove Theorem \ref{thm:An}, and to Viktor Ginzburg for explaining them in detail an argument in \cite[Lemma 3.4]{Gi} that is needed for the proof of \Cref{prop:embR8}.

\section{Preliminaries}\label{s:pre}

We recall in this section the basic definitions required throughout this work, as well as some notions in dynamics that will play an important role.

\subsection{Stable hypersurfaces and Hamiltonian structures}

Let $M$ be an oriented manifold of dimension $2n+1$.
\begin{defi}
    A \emph{Hamiltonian structure} $\omega$ is a closed two-form of maximal rank in $M$.
\end{defi}
In general, one can always find a one-form $\lambda$ such that $\lambda\wedge \omega^n$ induces the given orientation on $M$, which we shall denote in short by $\lambda\wedge \omega^n>0$ as usual; the pair $(\lambda,\omega)$ is known as a framed Hamiltonian structure. Stability requires the existence of a one-form with an additional property.
\begin{defi}
    A \emph{stable Hamiltonian structure} in a manifold $M$ of dimension $2n+1$ is a pair $(\lambda,\omega)$ where $\omega$ is a Hamiltonian structure and $\lambda$ is a one-form satisfying $\lambda\wedge \omega^n>0$ and $\ker \omega \subset \ker d\lambda$.
\end{defi}
We say in this case that $\omega$ is \emph{stabilizable}, and $\lambda$ is a stabilizing one-form for $\omega$. 
A stabilizable Hamiltonian structure admits several stabilizing one-forms in general, even besides constant multiples of a given $\lambda$. 
A \shs determines a vector field $R$ called the \emph{Reeb field} of $(\lambda,\omega)$ via the equations
$$\begin{cases}
    \lambda(R)=1,\\
    \iota_R\omega=0.
\end{cases}$$
In the second part of this paper, we will be interested in stable Hamiltonian structures up to stable homotopy.

\begin{defi}
A \emph{stable homotopy} is a homotopy of stable Hamiltonian structures $(\lambda_t,\omega_t)$ such that the cohomology class of $\omega_t$ remains constant.
\end{defi}
If we identify all the stable Hamiltonian structures homotopic to a given one, we obtain a \emph{stable homotopy class}. 
Notice that the stable homotopy class of a \shs $(\lambda,\omega)$ is determined by the stabilizable Hamiltonian structure $\omega$, since any two stabilizing one-forms $\lambda_1, \lambda_2$ induce homotopic stable Hamiltonian structures $(\lambda_1,\omega)$ and $(\lambda_2,\omega)$. 
We can thus unambiguously speak of the stable homotopy class of a stabilizable Hamiltonian structure. \\

The relation between stable Hamiltonian structures and symplectic manifolds can be summarized as follows. 
An embedded hypersurface $M$ in a symplectic manifold $(W,\Omega)$ inherits a Hamiltonian structure induced by the ambient symplectic form. 
When the Hamiltonian structure is stabilizable, the hypersurface is called ``stable".
\begin{defi}
    A hypersurface $M$ in a symplectic manifold $(W,\Omega)$ is \emph{stable} if the Hamiltonian structure $\omega=i^*\Omega$ is stabilizable, where $i:M\hookrightarrow W$ denotes the inclusion map of $M$ into $W$.
\end{defi}
The Reeb field of $(\lambda,\omega)$ for any stabilizing one-form of $\omega$ integrates to the characteristic foliation of $M$. As first observed in \cite{EKP}, being stable is equivalent to the existence of a neighborhood $M\times (-\varepsilon,\varepsilon)$ in $W$ for which the characteristic foliation of $M\times \{t\}$ is diffeomorphic to that of $M$ for each $t\in (-\varepsilon,\varepsilon)$.\\

For a \shs $(\lambda,\omega)$ on a three-dimensional manifold $M$, the fact that $\ker \omega \subset \ker d\lambda$ implies that $\d\lambda=f\omega$ for some function $f\in C^\infty(M)$. This function is always a first integral of the Reeb field, since $\iota_R (\d f\wedge \omega)=\d f(R)\cdot\omega$ vanishes as $\d f \wedge \omega = \d^2\lambda =0$; in particular, $\d f(R)=0$. 
For a manifold $M$ of dimension $2n+1$, for each $i=1,..,n$ the differential $2n$-form $\d\lambda^{n-i}\wedge\omega^i$ has the Reeb vector field $R$ of $(\lambda,\omega)$ in its kernel, and hence one can write $\d\lambda^{n-i}\wedge\omega^i = f_i \omega^n$ for some function $f_i\in C^\infty(M)$.
Again, taking the exterior derivative on both sides we deduce that $\iota_R(\d f_i \wedge \omega^n)=0$, i.e.\ that $\d f_i(R)=0$. 
In other words, we have just proved the following:
\begin{lemma}
\label{lem:first_integrals}
For every $1\leq i \leq n$, the functions $f_i\colon M\to \R$ such that $\d\lambda^{n-i}\wedge\omega^i = f_i \omega^n$ are (possibly trivial) first integrals of the Reeb vector field $R$ of $(\lambda,\omega)$.
\end{lemma}

\medskip

\subsection{Inputs from dynamics}

We introduce here results in dynamics that will be used in our proofs. 

\medskip

\paragraph{\textbf{Normally hyperbolic invariant manifolds.}}
The first important tool that we will need is that of normally hyperbolic invariant submanifolds.

\begin{defi}
Let $X$ be a vector field on a manifold $M$ with flow $\varphi_t:M\longrightarrow M$. 
An invariant submanifold $N\subset M$ of $X$ is \emph{normally hyperbolic} if there exist constants $\rho_0>0$, $\rho_+> \rho_0$, $\rho_- >\rho_0$ and $C>0$ such that the following hold:
\begin{itemize}
    \item[-] there is a continuous splitting $TM|_N=TN\oplus E^+ \oplus E^-$, invariant under the flow and inducing a splitting $D\varphi_t|_N=D\varphi_0^t \oplus D\varphi_+^t \oplus D\varphi_-^t$;
    \item[-] $\norm{D\varphi_0^t(v)} \leq Ce^{\rho_0|t|}\norm{v}$, for all $t\in \mathbb{R}$ and $v\in TN$;
    \item[-]  $\norm{D\varphi_+^t (v)} \geq Ce^{\rho_+t}\norm{v}$, for all $t\geq 0, p\in N, v\in E^+$;
    \item[-] $\norm{D\varphi_-^t (v)} \leq Ce^{-\rho_-t}\norm{v}$, for all $t\geq 0, p\in N, v\in E^-$.
\end{itemize}
In addition, we say that $N$ is \emph{$r$-normally hyperbolic}, with $r>0$ an integer, if the second item above is replaced by $\norm{D\varphi_0^t(v)} \leq Ce^{\frac{\rho_0}{r}|t|}\norm{v}$, for all $t\in \mathbb{R}$ and $v\in TN$.
\end{defi}
Using this language, normal hyperbolicity is the same as $1$-normal hyperbolicity. 

The above notions are very standard, but let us at least mention that qualitatively what we are requiring with normal hyperbolicity is that the flow of $X$ leaves $N$ invariant and expands $E_+$, respectively contracts $E_-$, at a bigger rate than it expands, respectively contract, $TN$. 
For $r$-normal hyperbolicity, we ask that the rate of expansion and contraction in the normal directions dominates the ones in $TN$ by a factor $r$.

\medskip

The main well-known property of compact normally hyperbolic invariant submanifolds that we will use is that they are robust under perturbations of the vector field:
\begin{theorem}[\cite{Fe} and {\cite[Theorem 4.1]{HPS}}]
\label{thm:normal_hyperbolic_stable}
    Let $r\geq 1$ be an integer, and $X$ be a vector field with a compact $r$-normally hyperbolic invariant submanifold $N$. Any vector field $Y$ that is sufficiently $C^k$-close to $X$, with $k$ an integer such that $1\leq k \leq r$, admits a $r$-normally hyperbolic invariant $C^r$-submanifold $\widetilde N\cong N$, 
    that is $C^k$-close to $N$. 
\end{theorem}

\begin{Remark}\label{rem:NHSparam}
    The same notion can be defined for diffeomorphisms of a manifold, and analogous statements hold. A normally hyperbolic invariant submanifold persists under $C^1$-perturbations of the diffeomorphism. Furthermore, the invariant submanifold varies smoothly with the perturbation (see \cite[Remark 1 page 52]{HPS}). Hence if we have a smooth one-parameter family of diffeomorphisms $f_t, t\in [0,1]$ with a smooth one-parameter family of normally hyperbolic submanifolds $N_t, t\in [0,1]$, a small enough perturbation $\tilde f_t$ of this family admits a smoothly varying family of (normally hyperbolic) invariant submanifolds $\widetilde N_t$. 
\end{Remark}

\paragraph{\textbf{Rotation number of circle diffeomorphisms.}}

The material we now recall is standard; see for instance \cite{Kat} for more details.
Let $f:S^1\longrightarrow S^1$ be a diffeomorphism of the circle, and consider its lift $\tilde f: \mathbb{R} \longrightarrow \mathbb{R} $ by the covering map $\pi:\mathbb{R}\longrightarrow S^1$ given by $\pi(x)=e^{2\pi x i}$. That is $\tilde f$ satisfies $f\circ \pi= \pi\circ \tilde f$, and such a lift is unique up to adding an integer constant. Define
$$ \tau(\tilde f)=\lim_{n\rightarrow \infty} \frac{1}{n} (\tilde f^n(x)-x),$$
which exists for all $x$ and is independent of it. For any two lifts of $f$, these numbers differ only by integers. Thus, the following notion is well defined.
\begin{defi}
    The rotation number of $f$ is $\tau(f)=\pi(\tau(\tilde f))$.
\end{defi}
Of course, the rotation number of a rotation in $S^1$ is the angle of rotation. The rotation number can be understood as a point in $[0,1]$ (identifying the boundary points), it is invariant under conjugacy and satisfies the following two properties \cite[Proposition 11.1.4 and Proposition 11.1.6]{Kat} which will be of use to us.
\begin{prop}\label{prop:rotperiodic}
    Let $f:S^1\longrightarrow S^1$ be a diffeomorphism. 
    Then $\tau(f)\in \mathbb{Q}$ if and only if $f$ has a periodic orbit.
\end{prop}
\begin{prop}\label{prop:rotcont}
    The map $\tau: \operatorname{Diff}(S^1)\longrightarrow S^1$ is continuous when we equip $\operatorname{Diff}(S^1)$ with the $C^0$-topology.
\end{prop}

\medskip

\section{Non-density of stable hypersurfaces}\label{s:T1}

In this section, we establish the non-density of stable hypersurfaces in dimensions greater than eight, as stated in Theorem \ref{thm:main1}. The first step is to construct in Theorem \ref{thm:An} a Hamiltonian structure on a 3-dimensional homology sphere that is robustly non-stabilizable. Thanks to the specific properties of this Hamiltonian structure, we show that it can be embedded as an invariant submanifold of the characteristic foliation of any hypersurface of dimension at least eight (Proposition \ref{prop:firstemb}) with a suitable normal form that allows making it a normally hyperbolic invariant submanifold (Proposition \ref{prop:addhyper}). This invariant submanifold acts as a $C^2$-robust obstruction to the stability of the hypersurface.

\subsection{A robustly non-stable Anosov flow on a 3D homology sphere}\label{ss:An}

In this subsection, we will prove the existence of a suitable three-dimensional non-vanishing and volume-preserving vector field that spans the kernel of a Hamiltonian structure that is robustly non-stabilizable.

\begin{defi}
    A vector field $X$ on a three-dimensional manifold $N$ is Anosov if there exists constants $C, \rho>0$ and a continuous splitting $TN=\langle X\rangle \oplus E_u \oplus E_s$, with $E_u, E_s$ of constant rank $1$, such that:
\begin{itemize}
    \item[-] $\norm{(\phi_t)_*v} \leq Ce^{-\rho t}\norm{v}$, for all $t\geq 0$ and $v\in E_s$;
    \item[-] $\norm{(\phi_t)_*v} \geq Ce^{\rho t}\norm{v}$, for all $t\geq 0$ and $v\in E_u$.
\end{itemize}
\end{defi}
The splitting is in general only of $C^0$-regularity. Such flows are $C^1$-structurally stable: any flow that is sufficiently $C^1$-close to $X$ is orbit equivalent to $X$. As shown by Anosov \cite{An}, the distributions $E_s, E_u, E_s\oplus \langle X \rangle, E_u \oplus \langle X\rangle$ are all integrable. They define $C^0$-foliations denoted by $\mathcal{F}_{ss}, \mathcal{F}_{uu}, \mathcal{F}_s, \mathcal{F}_u$ and called strong stable, strong unstable, stable and unstable foliations respectively.

In three dimensions, Anosov flows on closed three-manifolds fall into three disjoint categories \cite{F, B1}: flat, skewed, and non $\mathbb{R}$-covered. We don't give here the precise definitions of these terms, but just limit ourselves to mentioning that each of these categories is determined by the orbit equivalence class of the Anosov flow.
More precisely, we will need two facts in the proof of the lemma below: an Anosov flow orbit equivalent to a suspension is flat (see e.g \cite{F}), and an Anosov flow in the kernel of a contact type $2$-form $d\alpha$ (i.e. $\alpha$ satisfies $\alpha\wedge d\alpha\neq 0$ at every point) is skewed \cite{B2}.

\begin{lemma}\label{lem:nonstab}
    Let $X$ be a smooth Anosov flow preserving a smooth volume form $\mu$ on a closed three-manifold. If $X$ is non $\mathbb{R}$-covered, then $\iota_X\mu$ admits no stabilizing one-form. 
    Moreover, any Hamiltonian structure $\omega$ that is $C^1$-close to $\iota_X\mu$ is not stabilizable.
\end{lemma}
\begin{proof}
    We start with the first part of the statement: assuming that $\lambda$ is a stabilizing one-form for $\omega=\iota_X\mu$, we will prove that $X$ is either flat or skewed.
    Thanks to the existence of the stabilizing one-form, we know that $d\lambda=f\omega$ for some function $f\in C^\infty(M)$ which is necessarily a first integral of $X$. 
    However, it is well-known that an Anosov flow admits no non-trivial first integral, and thus $f$ is necessarily constant. Now, if $f\neq 0$ then $\lambda$ is a contact form; as the vector field $X$ is in the kernel of the contact type two-form $d\lambda$, it must be skewed. 
    If on the contrary $f\equiv 0$, it follows via an argument of Tischler \cite{Ti} that the flow admits a global cross-section and thus that it is orbit equivalent to a suspension. 
    This implies that $X$ is flat, thus concluding the proof of the first part of the lemma.
    
    As far as the second part is concerned, notice that, by the structural stability of $X$, the vector field $Y$ defined by $\iota_Y\mu=\omega$ will be a non $\mathbb{R}$-covered Anosov flow as well. This, the two-form $\omega$ admits no stabilizing one-form by the first part of the lemma.
\end{proof}

\begin{Remark}\label{rem:reguAnosov}
     In Lemma \ref{lem:nonstab}, the two-form $\omega$ can be allowed to have $C^1$-regularity, and the stabilizing one-forms to be of $C^2$-regularity.
     Indeed, in this case, the integral $f$ is of $C^1$-regularity, and the argument goes through.
\end{Remark}

\medskip

For our construction, we need to first recall some facts about surgery operations on Anosov flows. 

The \emph{Dehn-Goodman-Fried surgery} of an Anosov flow $X$ in a manifold $M$ is a Dehn surgery with certain coefficients along a periodic orbit (or a collection of periodic orbits) of $X$ that transforms $(M,X)$ into another pair $(M',X')$ where $X'$ is again an Anosov flow. This operation was initially defined in two different ways by Fried \cite{Fr} and Goodman \cite{Go}, and then it was proven by Shannon \cite{Sh} that these two definitions coincide. We list the properties about this surgery operation that we will need:
\begin{enumerate}
    \item if $X$ is transitive, then so is $X'$ \cite[Chapter 3]{Sh}.
    \item Given any transitive Anosov vector field, it is shown in \cite{As} that we can find a vector field that is orbit equivalent to it which is smooth and that preserves a smooth volume form. This applies in particular to any $X'$ obtained from a surgery of a transitive Anosov flow, which will hence be assumed to be smooth and volume-preserving.
    \item Assuming that $M$ is oriented, if $X$ has coorientable stable and unstable foliations, then so does $X'$ (this just follows from the description of the surgery, see e.g. \cite{Sh}).
\end{enumerate}

We proceed to construct a non $\mathbb{R}$-covered Anosov flow with coorientable stable and unstable foliations on an integral homology sphere. (These properties will be needed for later use in Proposition \ref{prop:embR8}.) The strategy for this construction has been suggested to us by Jonathan Bowden.

\begin{theorem}\label{thm:An}
There exists a three-dimensional integral homology sphere $N$ and a smooth volume-preserving Anosov flow $X\in \mathfrak{X}(N)$ with coorientable stable and unstable foliations that is non $\mathbb{R}$-covered.
\end{theorem}

\begin{proof}
By \cite[Theorem A and Remark 2.2]{De} there exists an Anosov flow $Y_1$ in some integral homology sphere $N_1$ whose strong stable foliation $\mathcal{F}_{ss}$ and strong unstable foliation $\mathcal{F}_{uu}$ are coorientable. 
The flow corresponds to the geodesic flow of a hyperbolic orbifold. 
Let now $Y_2$ be the flow obtained by the suspension of the cat map on $T^2$.
This yields an Anosov flow in a torus bundle $N_2$ over the circle. 

By \cite[Theorem A]{DS}, the flows $(N_1,Y_1)$ and $(N_2,Y_2)$ are \emph{almost equivalent}, i.e. there exist finite collections of periodic orbits $C_1$ and $C_2$ of the flows $Y_1$ and $Y_2$ respectively, such that 
$$N_1\setminus C_1 \cong N_2\setminus C_2,$$
and $Y_1, Y_2$ are orbit equivalent away from these periodic orbits. 
As shown in \cite{Sh}, almost equivalence can also be characterized by the fact that $(N_1, Y_1)$ can be obtained from $(N_2,Y_2)$ via finitely many Dehn--Goodman--Fried surgeries. 
Let then $\mathcal{C} \subset N_2$ be the union of the finitely many periodic orbits where these surgeries are performed. 

It is moreover shown in \cite[Theorem 6]{BI} that there exist two periodic orbits $\gamma_+,\gamma_-$ of $Y_2$ in $N_2$, that can be chosen in the complement of $\mathcal{C}$, such that doing a negative surgery on $\gamma_-$, a positive surgery on $\gamma_+$ and any surgery on the orbits of $\mathcal{C}$ always yields a non $\mathbb{R}$-covered Anosov flow.
Thus, if we do the Dehn--Goodman--Fried surgeries along $\mathcal{C}$ that bring $(N_2,Y_2)$ to $(N_1,Y_1)$, and any positive and negative surgery on $\gamma_+$ and $ \gamma_-$, we will always obtain a non $\mathbb{R}$-covered Anosov flow $X$ on some three-manifold $N$. 
By the three properties of this surgery recalled above, the flow $X$ has coorientable stable and unstable foliations, and can be assumed to be smooth and to preserve a smooth volume form. 
Choosing as coefficients $+1$ and $-1$ for the surgeries in $\gamma_+$ and $\gamma_-$ respectively, the integral homology of the manifold $N$ can moreover be arranged to be the same as that of $N_1$, see e.g. \cite[Section 1.1.5]{Sa}. Hence, we obtain a smooth volume-preserving and non $\mathbb{R}$-covered Anosov flow $X$ with coorientable stable and unstable foliations on an integral homology sphere $N$.
\end{proof}
In fact, we will only need that the ambient manifold is a rational (rather than integral) homology sphere. 
Indeed, for our purposes this already guarantees that any Hamiltonian structure having $X$ in its kernel, hence in particular $\iota_X\mu$ where $\mu$ is the smooth volume form preserved by $X$, is exact.

\medskip

\subsection{An invariant submanifold with Anosov dynamics}

We now show that given a hypersurface $M$ on a symplectic manifold of dimension at least eight, there is a $C^0$-perturbation of $M$ whose characteristic foliation contains an invariant three-dimensional copy of $N$, with a trivial symplectic normal bundle, where the Reeb dynamics is exactly given by the Anosov flow constructed in Theorem \ref{thm:An}.

\smallskip

Consider the Anosov flow $X$ on a rational homology sphere $N$ constructed in Theorem \ref{thm:An}, and let $\mu$ be the smooth volume form preserved by $X$. Denote the two-form $\iota_X\mu$ by $\widetilde \omega$ (which is exact by homological reasons), and choose a one-form $\alpha$ such that $\alpha(X)=1$. Endow $V=N\times (-\varepsilon,\varepsilon)$ with the form $\omega_V= \widetilde \omega + d(t\alpha)$, where $t$ is the coordinate in $(-\varepsilon,\varepsilon)$.
It is non-degenerate for $\varepsilon>0$ small enough. 
We will show that $V$ embeds symplectically in a $\mathbb{R}^7$-slice of the standard symplectic $\mathbb{R}^8$.
Embedding certain symplectic manifolds in a hyperplane of the standard symplectic space is a strategy already used in \cite{Gi} as a method to embed dynamics in hypersurfaces. 
The difficulty lies in the fact that there are in general some formal obstructions to construct such an embedding. 
The construction of such an embedding in our case will then of course use the properties of the Anosov flow constructed in \Cref{thm:An}. 

Using the $h$-principle language from e.g.\ \cite{hprinc}, we recall that the formal counterpart of an isosymplectic immersion of a symplectic manifold $(V,\omega_V)$ into another symplectic manifold $(W,\omega_W)$ is an \emph{isosymplectic monomorphism}, i.e. a monomorphism
$$F: TV \longrightarrow TW,$$
satisfying $\omega_V(u,v)=\omega_W(F(u),F(v))$ for each $u,v\in TV$, and such that the base map $f:= \operatorname{bs}F: V\rightarrow W$ satisfies $f^*[\omega_W]=[\omega_V]$.

\begin{prop}\label{prop:embR8}
    For some small enough $\delta>0$, the symplectic manifold $(V=N\times (-\delta,\delta),\omega_V= \widetilde{\omega} + \d(t\alpha))$ admits an isosymplectic embedding into $(\mathbb{R}^8, \omega_{std})$ and a tubular neighborhood $U\cong V \times (-\delta,\delta)^4$ such that 
    $$\omega_{std}|_U= \omega_V + \sum_{i=1}^2 du_i\wedge dv_i, $$
    where $(u_1,v_1,u_2,v_2)$ are coordinates in $(-\delta,\delta)^4$. 
    We can further assume that $v_2=y_4|_{U}$, where $(x_i,y_i)_{i=1,\ldots,4}$ are the standard symplectic coordinates in $\mathbb{R}^8$.
\end{prop}

\begin{proof}
Let $X$ be the Anosov flow on an integral homology sphere $N$ constructed in Theorem \ref{thm:An}, which lies in the kernel of the exact Hamiltonian structure $\widetilde \omega$. Denote by $\omega_{std,6}$ the standard symplectic form in $\mathbb{R}^6$, and the same notation will be used for $\pi^*\omega_{std,6}$ in $\mathbb{R}^7$, where $\pi$ is the projection to the hyperplane $\mathbb{R}^6 \times \{0\}$. Our first goal will be to construct an embedding $e$ of $V'=N\times (-\delta,\delta)$ into $\mathbb{R}^7$, for a small enough $\delta>0$, such that $e^*\omega_{std,6}=\omega_V|_{V'}$. Such an embedding is what Ginzburg \cite{Gi} calls a ``symplectic" embedding of $(V',\omega_V|_{V'})$ into $(\mathbb{R}^7, \omega_{std,6})$.

By construction, $X$ has cooriented strong stable and strong unstable foliations $\cF_{ss}$ and $\cF_{uu}$. 
This implies that any plane field on $N$ in the complement of $\langle X \rangle$ can be trivialized. Choose then a non-vanishing vector field $Y$ in $\eta=\ker\alpha$.
Let also $J$ be an almost complex structure on $V$ with $J\eta=\eta$ and that is tamed by $\omega_V$.
Then, we have a complex trivialization $TV = \langle Y,JY,X,JX\rangle$.
(Here, we naturally identified $Y$ and $X$ in $N$ with $(Y,0)$ and $(X,0)$ in $N\times (-\varepsilon,\varepsilon)$.)
In terms of the symplectic structure on $TV\to V$, this just means that $(TV,\omega_V)$ is a trivial symplectic bundle.

By the Whitney immersion theorem, the manifold $N$ admits an immersion into $\R^5$.
Hence, one can also find an immersion $f:V\rightarrow \mathbb{R}^6$ just by thickening the first immersion.

Now, $\mathbb{R}^6$ is naturally endowed with the standard symplectic structure, which we denoted by $\omega_{std,6}$ to avoid confusion with the standard symplectic structure $\omega_{std}$ on $\mathbb{R}^8$ in the statement. The tangent bundle $(T\R^6,\omega_{std,6})$ is a symplectically trivial vector bundle.  
In particular, as both $(TV,\omega_V)$ and $(T\R^6,\omega_{std,6})$ are symplectically trivial bundles, one can find a monomorphism 
$$F:TV\longrightarrow T\mathbb{R}^6,$$
whose base map is $f$ and satisfying $\omega_V(\hat u, \hat v)=\omega_{std,6}(F(\hat u),F(\hat v))$ for any two $\hat u, \hat v\in TV$. 
Furthermore, as both $\omega_{std,6}$ and $\omega_V$ are exact, we trivially have $f^*[\omega_{std,6}]=[\omega_V]$.
In other words, the pair $(f,F)$ is an isosymplectic monomorphism.

The $h$-principle for isosymplectic immersions \cite[Section 3.4.2]{Gr} (c.f.\ \cite[Theorem 24.4.3]{hprinc} for a statement using our notation) then implies that there exists an isosymplectic immersion $\tilde f$ of $(V,\omega_V)$ into $(\mathbb{R}^6,\omega_{std,6})$, which is homotopic to $f$ among isosymplectic monomorphisms. 
We now claim that the normal bundle $\nu$ of the immersion $\tilde f$ is symplectically trivial.

For this, note that, by construction of $\tilde f$ via the $h$-principle, the pullback bundle $\tilde f^*T\mathbb{R}^6$ is isomorphic to the symplectic vector bundle $TV\oplus \nu_F$, where $\nu_F$ is the rank $2$ normal bundle of the symplectic monomorphism $F$.
Note also that $TV\oplus \nu_F= T\R^6$ symplectically.
Now, as previously remarked, $(TV,\omega_V)$ is a symplectically trivial bundle.
In particular, $\nu_F$ is a symplectically trivial bundle as well. (For instance, one can easily see that its first Chern class is zero.) 
Hence, $\nu$ is also symplectically trivial, as claimed. 

\medskip

It now follows from a claim in \cite[Lemma 3.4]{Gi} that $\tilde f$ can be perturbed to another isosymplectic immersion of $(V, \omega_V)$ such that $\tilde f|_{N\times \{0\}}$ has only transverse self-intersections, i.e. double points.
For completeness, we give a detailed proof of this claim, which has been kindly explained to us by V.\ Ginzburg. 

First, we cover $N$ by a finite number of balls $U_i$ such that $\tilde f|_{U_i\times (-\delta,\delta)}$ is an embedding for some $\delta>0$ small enough. 
Choose also some slightly larger sets $U_i'$ containing the closure of $U_i$ in its interior and satisfying as well that $\tilde f|_{U_i'\times (-\delta,\delta)}$ is an embedding. 
Taking the $U_i, U_i'$ small enough, one can always find coordinates $(r_1,s_1,r_2,s_2,r_3,s_3)$ on a neighborhood $W_i\subset \mathbb{R}^6$ of $\tilde f(U_i'\times (-\delta,\delta))$ such that 
\begin{itemize}
    \item[-] $\tilde f(U_i'\times (-\delta,\delta))\cap W_i = \{r_3=s_3=0\}$ (this is ensured by the Weinstein tubular neighborhood theorem applied to the symplectic submanifold $\tilde f(U_i'\times (-\delta,\delta))\cap W_i$);
    \item[-] $\tilde f(U_i')\cap W_i = \{r_3=s_3=s_2=0\}$ (this holds by choosing Hamiltonian flow-box coordinates of the hypersurface $\tilde f(U_i')$ inside $\tilde f(U_i'\times (-\delta,\delta)$).
\end{itemize}
Since $\tilde f$ is an embedding along $U_i'\times (-\delta,\delta)$, the self intersections of the image of $\tilde f\vert_{N\times\{0\}}$ that are contained inside $W_i$ must involve at most one point in $U_i\times\{0\}$ in the source.
Notice at this point that, by standard transversality arguments, a generic translation of $\tilde f(U_i')\cap W_i$ in the normal directions $s_2,r_3,s_3$ will put it in general position with respect to $\tilde f((N\setminus U_i') \times \{0\})$. 
Choose one of these translations (which can moreover be chosen arbitrarily close to the identity); this is an ambient Hamiltonian diffeomorphism of $W_i$, so one can cut off the generating Hamiltonian by a function that is equal to $1$ in $U_i$ and $0$ close to the boundary of $U_i'$ and more generally of $W_i$. 
Doing so we can construct an (arbitrarily $C^\infty$-close to the identity) compactly supported Hamiltonian diffeomorphism $\phi_1$ of $W_i$ which restricts to $U_i$ as the required translation. 
In particular, the embedding $h=\phi_1\circ \tilde f|_{U_i}\times (-\delta,\delta)$ extends as an isosymplectic immersion $\tilde f' (V,\omega_V) \longrightarrow \mathbb{R}^6$ equal to $\tilde f$ away from $U_i'\times (-\delta,\delta)$ and such that $\tilde f'(U_i\times\{0\})$ intersects transversely $\tilde f'((N\setminus U_i')\times \{0\})$, and hence all self intersections of $\tilde f'$ involving at least one point in $U_i$ are transverse. 
Doing this perturbation first along $U_1$, we obtain an isosymplectic immersion $\tilde f_1$ such that $\tilde f_1|_{N\times \{0\}}$ only has transverse self-intersections along $U_1$.
We repeat the process in $U_2$ to construct another immersion $\tilde f_2$ such that the self-intersections along $N\times \{0\}$ are transverse along $U_2$, and choosing the translation (and hence the Hamiltonian diffeomorphism) close enough to the identity the self-intersections will remain transverse also along $U_1$. 
Doing this inductively for each $U_i$, we end up with the required isosymplectic immersion, which we still denote by $\tilde f$.

Since the self-intersections of $\tilde f$ along $N\times \{0\}$ are only double points, we can undo these intersections in $\mathbb{R}^7$ by slightly pushing off the immersion in the last coordinate of $\mathbb{R}^7$ near one of the two points mapped to each double point.
We obtain in this way an embedding $e: N\times (-\delta,\delta) \rightarrow \mathbb{R}^7$, for a small enough $\delta>0$, satisfying $e^*\omega_{std,6}=\omega_V|_{N\times(-\delta,\delta)}$.
To simplify the notation, we rename $N\times (-\delta,\delta)$ as $V$ and the restriction of $\omega_V$ to $N\times (-\delta,\delta)$ as $\omega_V$. 
Seeing $\mathbb{R}^7$ as $\mathbb{R}^7\times \{0\}$ in $\mathbb{R}^8$, we have thus an embedding into $\R^8$, that we still denote by $e$, such that $e^*\omega_{std}=\omega_V$, which admits a trivial symplectic normal bundle $\nu'=\nu \oplus \langle \pp{}{u_2},\pp{}{v_2}\rangle$, where $u_2,v_2$ denote the last two coordinates of $\mathbb{R}^8$.
Note that $v_2$ can be chosen to be, without loss of generality, equal to the last coordinate $y_4$ of $\R^8$ that satisfies $\R^7=\{y_4=0\}\subset \R^8$.
By construction of the embedding $V\hookrightarrow \R^7$, via standard symplectic normal form theorems, one can find symplectic coordinates $u_1,v_1$ in a small neighborhood $e(V)$ that span the fibers of the (trivial) bundle $\nu$.
A neighborhood of $e(V)$ is then symplectomorphic to $U=V\times (-\delta,\delta)^4$ with symplectic form $\omega_V+\sum_{i=1}^2 du_i\wedge dv_i$, as claimed.
\end{proof}

\begin{Remark}\label{rem:anydim}
Notice that by considering the trivial symplectic product $\mathbb{R}^8\times \mathbb{R}^{2n-8}$, with $2n>8$, and renaming coordinates it follows that there is an isosymplectic embedding of $(V,\omega_W)$ into $(\mathbb{R}^{2n}, \omega_{std})$ and a tubular neighborhood $U\cong V\times (-\delta,\delta)^{2n-4}$ with coordinates $(u_i,v_i)$ such that $\omega_{std}|_U= \omega_V + \sum_{i=1}^{n-2} du_i \wedge dv_i$, with $v_{n-2}=y_n$.
\end{Remark}

\begin{Remark}\label{rem:dim6}
    In order to improve Theorem \ref{thm:main1} to cover ambient symplectic manifolds of dimension six, the only missing step is constructing an isosymplectic embedding of the symplectic manifold $(V,\omega_V)$ into $\mathbb{R}^6$ as in Proposition \ref{prop:embR8} (i.e. with trivial normal bundle and lying on a hyperplane). 
    A similar situation arose in Ginzburg's Hamiltonian counterexamples to the Seifert conjecture, which were first constructed in $\mathbb{R}^8$ \cite{Gi}. 
    It was later that the missing isosymplectic embedding into $\mathbb{R}^6$ was constructed in \cite{Gi2}, implying the existence of counterexamples in $\mathbb{R}^6$.
    The manifold $(V,\omega_V)$ can be replaced by the symplectization of any three-manifold equipped with a Hamiltonian structure that is robustly non-stable and satisfies the required formal properties to admit an embedding in $\mathbb{R}^6$. 
    The only examples of such Hamiltonian structures that we know are (exact) non $\mathbb{R}$-covered volume-preserving Anosov flows with coorientable stable and unstable foliations.
\end{Remark}

We will now use the embedding to show that we can embed the manifold $N$ as an invariant submanifold of the characteristic foliation of an arbitrary hypersurface, in a way that the induced dynamics are the Anosov flow constructed in Theorem \ref{thm:An}. An embedding inducing some dynamics fixed a priori along a hypersurface in $W$ is called a \emph{Hamiltonian embedding} of $(N,X)$ into $W$, following the terminology introduced in \cite{CP}. 
We will keep track of the normal form of the symplectic structure in the normal bundle of the invariant submanifold.

\begin{prop}\label{prop:firstemb}
    Let $e:M\longrightarrow W$ be an embedded hypersurface in a symplectic manifold $(W,\Omega)$ with $\dim W=2n\geq 8$. 
    Then, there is a $C^0$-small isotopy $e_s: M\longrightarrow W$, compactly supported near a point and with $e_0=e$, and a subset $\widehat U \subset M$ such that 
    \begin{itemize}
        \item[-] $\widehat U\cong N \times (-\delta,\delta)^{2n-4}$,
        \item[-] ${e_1}^*\Omega|_{\widehat U}= \widetilde \omega + \sum_{i=1}^{n-2} \d u_i\wedge \d v_i$, 
    \end{itemize} 
    where $(u_i,v_i)$ are coordinates in $(-\delta,\delta)^{2n-2}$.
\end{prop}
\begin{proof}
  Choose a Hamiltonian flow-box neighborhood diffeomorphic to $U \subset \mathbb{R}^{2n}$ near a point $p\in M$ with coordinates $(x_i,y_i)$, such that 
  $$ \omega_{std}|_{U}= \sum_{i=1}^n \d x_i\wedge \d y_i,$$
  $$ e(M)\cap U=\{y_n=0\}. $$
  By Proposition \ref{prop:embR8} and Remark \ref{rem:anydim}, there exist an isosymplectic embedding $f:V=N\times (-\varepsilon,\varepsilon) \longrightarrow U$ of $(V,\omega_V)$ into $U$ and a neighborhood $U'\cong N\times (-\varepsilon,\varepsilon) \times (-\delta,\delta)^{2n-4}\subset U$ containing $f(V)$ such that 
  $$ \omega_{std}|_{U'}=\omega_V + \sum_{i=1}^{n-2} \d u_i \wedge \d v_i= \widetilde \omega + \d(t\alpha) + \sum_{i=1}^{n-2} \d u_i \wedge \d v_i,  $$
  where $t$ is the coordinate in $(-\varepsilon,\varepsilon)$ and $u_i,v_i$ are coordinates in $(-\delta,\delta)^{2n-4}$.
  Recall that Proposition \ref{prop:embR8} allows us to choose $f$ such that $v_{n-2}=y_n|_{U'}$. Consider the first $2n-5$ factors of $(-\delta,\delta)^{2n-4}$, which is diffeomorphic to a ball $B^{2n-5}$, where we use spherical coordinates $(\hat r,\hat \varphi_i)$, with $\hat r\in (0,\hat \delta)$ and $i=1,...,2n-6$. 
  The inclusion of $M\cap U'$ into $U'$ is given by
  \begin{align*}
      j_0: N\times (-\varepsilon, \varepsilon)\times B^{2n-5} &\longrightarrow N\times (-\varepsilon,\varepsilon) \times B^{2n-5} \times (-\delta,\delta)_{v_{n-2}}\\
        (p, \hat t, \hat r, \hat \varphi_i) &\longmapsto (p, \hat t,  \hat r, \hat \varphi_i,0),
  \end{align*} 
  which corresponds to $\{v_{n-2}=0\}$. Consider a different embedding of the form 
  \begin{align*}
      j_1: N \times (-\varepsilon, \varepsilon)\times B^{2n-5} &\longrightarrow N\times (-\varepsilon,\varepsilon) \times B^{2n-5} \times (-\delta,\delta)_{v_{n-2}}\\
        (p, \hat t, \hat r, \hat \varphi_i) &\longmapsto (p, f_1(\hat t,\hat r), \hat r, \hat \varphi_i ,f_2( \hat t,\hat r)),
  \end{align*} 
  where $f_1$ and $f_2$ are smooth functions from $(-\varepsilon,\varepsilon)\times (0,\hat \delta)$ satisfying
  \begin{itemize}
      \item[-] $f_1(\hat r, \hat t)=\hat t$ and $f_2(\hat r, \hat t)=0$ for $\hat r$ close to $\hat \delta$,
      \item[-] $f_1(\hat r, \hat t)= 0$ and $f_2(\hat r, \hat t)=\hat t$ for $\hat r$ near $0$ and $\hat t \in (-\varepsilon', \varepsilon')$ for some $\varepsilon'<\varepsilon$.
  \end{itemize}
  
  In fact, we can choose a family of embeddings $j_s$ that interpolates between $j_0$ and $j_1$, just by considering a family of embedded curves $(f_1^s,f_2^s)$ interpolating between $(\hat t, 0)$ and $(f_1,f_2)$.
  It induces a family of embeddings of $e_s:M \longrightarrow W$, by considering the trivial extension of $j_s$ as $e$ away from $N\times (-\varepsilon,\varepsilon)\times B^{2n-5}$. 
  Observe now that in $\widehat U=U' \cap \{r<\tau\}$ for a small enough $\tau<\delta$, we have $M\cap \widehat U=\{t=0\}$. In particular, by construction we have $e_1^*\omega_{std}|_{\widehat U}=\widetilde \omega + \sum_{i=1}^{n-2}\d u_i\wedge \d v_i$ as claimed.
\end{proof}

\subsection{Introducing normal hyperbolicity}

Observe that as a consequence of Proposition \ref{prop:firstemb}, we can perturb $M$ into another hypersurface $\widetilde M$ whose characteristic foliation admits a three-dimensional invariant submanifold where the dynamics correspond to a non $\mathbb{R}$-covered Anosov flow.
We want this invariant submanifold to persist under perturbations, i.e. we need it to be a normally hyperbolic invariant submanifold. 

\begin{prop}\label{prop:addhyper}
    Let $e:M\longrightarrow W$ be an embedded hypersurface in a symplectic manifold $(W,\Omega)$ of any even dimension. 
    Assume that there is an odd-dimensional compact submanifold embedded via $f:N \longrightarrow M$ of codimension $2k$, with trivial tubular neighborhood $U\cong N\times D^{2k}$, such that 
        $$e^*\Omega|_U= \widetilde \omega + \sum_{i=1}^k dx_i \wedge dy_i,$$
    where $\widetilde \omega$ is a Hamiltonian structure on $N$. 
    For any $r\geq 1$ integer, there is a $C^1$-close embedding $e_1: M\longrightarrow W$ such that $(e_1\circ f)^*\Omega=\widetilde \omega$ and $e_1\circ f(N)$ is a $r$-normally hyperbolic invariant submanifold of the characteristic foliation of $e_1(M)$.   
\end{prop} 

Here, by ``$r$-normally hyperbolic invariant submanifold of the characteristic foliation'' we mean a submanifold that is a union of ($1$-dimensional) leaves of the latter, and $r$-normally hyperbolic for any vector field directing said characteristic foliation.

\begin{proof}
    Fix a volume form $\mu$ in $M$ that restricts to a split volume form $\mu_N\wedge dx_1\wedge dy_1 \wedge ... \wedge dx_k \wedge dy_k$ on $U=N\times D^{2k}_{\delta}$ (where $\delta$ denotes the radius of the disk), for some volume form $\mu_N$ in $N$. 
    Let $X$ be the vector field defined by  $\iota_X\mu= (\omega_M)^n$, where $\omega_M=e^*\Omega$, and $Y$ be the vector field in $N$ defined by $\iota_Y\mu_N= \widetilde \omega^{n-k}$. Notice that $X|_U=(Y,0)$, and in particular $N= N\times \{0\}\subset U$ is an invariant submanifold of $\ker \omega_M$. Fix some $\alpha \in \Omega^1(U)$ such that $\alpha(X)=1$; due to the fact that $X\vert_U = Y$, we can moreover choose it so that it is independent of the $D^{2k}_\delta$ factor, and more precisely so that $\alpha \in \Omega^1(N)$. 
    For a constant $C>0$ to be determined later, consider then the two-form 
    $$\widehat \omega_U=\omega_M|_U + C\sum_{i=1}^k d(x_iy_i\alpha).$$

    For any fixed $C$, we can shrink $U$ enough so that $\eta=C \cdot \sum_{i=1}^k x_iy_i\alpha$ is arbitrarily $C^1$-small in $U$, and thus $\widehat \omega_U$ defines a Hamiltonian structure on $U$. 
    Notice that $\ker \widehat \omega_U$ coincides along $N$ with $\ker \omega_M|_U$ and thus with $\ker \widetilde \omega$. Hence $N$ is an invariant submanifold of the characteristic foliation of $\widehat \omega_U$ too. 
    First, we will show that we can perturb the embedding of $M$ to another embedding $e_1:M\longrightarrow W$ such that $e_1^*\Omega|_U=\widehat \omega_U$ (after possibly shrinking $U$), and then we will show that $e_1\circ f(N)$ is a normally hyperbolic invariant submanifold of the characteristic foliation of $e_1(M)$; this will conclude the proof. \\
    
    To argue the existence of a possibly smaller neighborhood $U'=N\times D^2_{\delta'} \subset U$ and an embedding $e_1:M \longrightarrow W$ such that $e_1^*\Omega|_{U'}=\widehat \omega_U$, we will adapt \cite[Lemma 49]{C}. 
    Consider a tubular neighborhood $V\cong U\times (-\varepsilon,\varepsilon)$ of $U$ inside $W$ where $\Omega$ takes the form
    $$ \Omega_0=d(t\alpha)+\omega_M, $$
    where $t$ is a coordinate in $(-\varepsilon,\varepsilon)$. Up to choosing $\varepsilon, \delta'$ small enough, the path 
    $$\Omega_s=d(t\alpha)+(1-s)\omega_M + s \widehat \omega_M$$
    is made of two-forms which are symplectic on $V$. 
    Notice also that if $\iota :U \times \{0\}\rightarrow V$ is the inclusion of the zero section, we have 
    \begin{equation}\label{eq:zeroS}
        \iota^*\Omega_1=\widehat \omega_M.
    \end{equation} Observe as well that $\Omega_s-\Omega_0=sKd\eta$, and $\eta$ is a one-form that vanishes along $TM|_N$. 
    Following Moser's path argument, we consider the vector field $Z_s$ determined by the equation
            $$\iota_{Z_s}\Omega_s=sK\eta. $$
    The vector field $Z_s$ integrates into a flow
    $$ \varphi_s:  V \longrightarrow V. $$
    Since $\eta|_N=0$, the vector field $Z_s$ vanishes along the compact submanifold $N\times \{0\}\subset V$, and thus the flow $\varphi_s$ is defined up to time one on a sufficiently small neighborhood $V'=N\times D^2_{\delta'}\times (-\varepsilon',\varepsilon')$ of $N$, where $\delta'<\delta$ and $\varepsilon'<\varepsilon$. 
    By construction, this flow satisfies $\varphi_s^*\Omega_s=\Omega_0$. 
    
    With a slight abuse of notation, for simplicity we denote again by $\varphi_s$ the restriction of $\varphi_s$ to $V'$. Let $r$ be a radial coordinate in the product $D^2_{\delta'}\times (-\varepsilon',\varepsilon')$ defined for $r\in (0,\tilde \delta)$ for some small enough $\tilde\delta$. Consider a smooth function $f:[0,\tilde \delta)\rightarrow [0,1]$ equal to $1$ near $r=0$ and equal to $0$ near $r=\tilde \delta$. Consider the family of vector fields $\widehat Z_s = f(r)Z_s$. 
    Its flow $\psi_s: V' \longrightarrow V$ defines a diffeomorphism from $V'$ onto its image. 
    Notice that because $Z_s$ vanishes along $N\times \{0\}$, if we take $f$ to vanish away of a sufficiently small neighborhood of $N\times \{0\}$, the vector field $\widehat Z_s$ is arbitrarily $C^0$-small and hence its flow is $C^0$-close to the identity. In other words:
    \begin{itemize}
        \item[-] $\psi_s$ has compact support contained in the interior of $V'$ and image contained in the interior of $V$,
        \item[-] it is is arbitrarily $C^0$-close to the identity,
        \item[-] there is a smaller neighborhood $\widetilde V = \widetilde U\times (-\tilde \varepsilon, \tilde \varepsilon)\subset V'$ such that $\psi_1|_{\widetilde V}=\varphi_1|_{\widetilde V}$.
    \end{itemize} 
    The hypersurface $M$ intersects $V$ along $U \times \{0\}$, i.e. $M$ corresponds to the inclusion of the zero section $\iota: U\times \{0\} \rightarrow V$. 
    Consider now the embedding
    $$ h= \psi_1\circ \iota: U\longrightarrow V. $$
    This coincides with $e$ near the boundary of $V$, hence in particular it extends naturally to an embedding 
    $$e_1: M\longrightarrow W, $$
    such that $e_1(M\setminus U)=e(M\setminus U)$. 
    Finally, we have
    $${e_1}^*\Omega|_{\widetilde U}= \iota^*{\psi_1|_{U}}^*\Omega|_{\widetilde U}=\iota^*\Omega_1= \widehat \omega_U, $$
    where we used Equation \eqref{eq:zeroS}. This establishes the existence of $e_1$.\\

    To conclude, let us show that for a suitable choice of the constant $C>0$, the characteristic foliation of $e_1(M)$ has $N$ as a normally hyperbolic invariant submanifold of its characteristic foliation. 
    We denote by $\widehat \omega_M$ the pullback two-form $e_1^*\Omega$. Consider the vector field
    $$Z= X +  \sum_{i=1}^k C x_i \pp{}{x_i} - C y_i \pp{}{y_i},$$
    on $U$. 
    For any a priori given $r$, by compactness we can choose a sufficiently large $C$ so that the submanifold $N\times \{0\}$ is a $r$-normally hyperbolic invariant submanifold of $Z$.

    Even if $Z$ does not span $\ker \widehat \omega_M$, we will see that it coincides up to second order with a vector field $\widehat Z$ spanning $\ker \widehat \omega_M$ along $N$. 
    This is enough to deduce that $\widehat Z$ has $N$ as a $r$-normally hyperbolic invariant submanifold. 
    Indeed, the fact that $Z$ and $\widehat Z$ coincide at the $0$-order implies that $N$ is also an invariant submanifold of $\widehat Z$. 
    In addition, since $r$-normal hyperbolicity only depends on the first derivatives of the vector field along the invariant submanifold, the fact that $Z$ and $\widehat Z$ coincide at first order along $N$ implies that $N$ is a $r$-normally hyperbolic invariant submanifold of $\widehat Z$, as desired. 
    \\

    Let us hence prove the existence of such a $\widehat Z$. Let $\beta$ be a one-form in $U$ such that $\beta(Z)=1$, and consider the volume form $\mu= \beta \wedge (\widehat \omega_M)^{n}$. Notice that 
    \begin{align*}
        \iota_Z \mu&= (\widehat \omega_M)^n + n\cdot \iota_Z \widehat \omega_M \wedge  (\widehat \omega_M)^{n-1} \wedge \beta,
    \end{align*}
    and since $\iota_Z \widehat \omega_M =C \sum_{i=1}^k x_iy_i \iota_Xd\alpha$, we deduce that
    \begin{align*}
        \iota_Z \mu&= (\widehat \omega_M)^n + n \left(C\sum_{i=1}^k x_iy_i \iota_Xd\alpha \right) \wedge (\widehat \omega_M)^{n-1} \wedge \beta,
    \end{align*}
    Let $\widehat Z$ be the vector field such that $\iota_{\widehat Z}\mu=(\widehat \omega_M)^n$, i.e. it spans the kernel of $\widehat \omega_M$.
    Notice that at any point $p\in N$, the two-forms $\iota_Z\mu$ and $\iota_{\widehat Z}\mu$ coincide up to second order in the variables $x_i,y_i$. 
    This implies that $Z$ and $\widehat Z$ coincide up to second-order terms in the normal direction of $N$.
    We deduce that not only $Z$ and $\widehat Z$ but also their respective linearizations coincide along $N$. 
    Thus the differential of their flows coincide along $N$ and the submanifold is also a $r$-normally hyperbolic invariant submanifold of $\widehat Z$. 
    This shows that $e_1\circ f(N)$ is a $r$-normally hyperbolic invariant submanifold of the characteristic foliation of $e_1(M)$.
\end{proof}

\subsection{Proof of Theorem \ref{thm:main1}} 
Let $M$ be a hypersurface in $(W,\Omega)$. 
Let $N$ be the $3$-dimensional homology sphere and $X$ the volume-preserving Anosov flow on $N$ given by \Cref{thm:An}.
Fix a volume form $\mu$ on $N$ preserved by $X$; we then have $\iota_X\mu=\widehat \omega$, which is a Hamiltonian structure that is necessarily exact due to the fact that $H^2(N;\R)=\{0\}$. 
By \Cref{prop:embR8,prop:addhyper}, there is a $C^0$-perturbation $\widetilde M$ of $M$, equipped with the Hamiltonian structure $\widetilde \omega$ obtained by pulling back $\Omega$ by the inclusion map, and an embedding $e:N\rightarrow \widetilde M$ such that $e^*\widetilde \omega=\widehat \omega$ and $e(N)$ is a $r$-normally hyperbolic invariant submanifold of the characteristic foliation of $\widetilde M$, where we choose any $r\geq 3$.

Let now $M'$ be a hypersurface in a sufficiently small $C^3$-neighborhood of $\widetilde M$ with induced Hamiltonian structure $\omega'$.
Then,
according to \Cref{thm:normal_hyperbolic_stable}, there is a $C^r$-embedding $f:N\rightarrow M'$ that is $C^2$-close to $e$ and such that $f(N)$ is a $r$-normally hyperbolic invariant submanifold of the characteristic foliation on the hypersurface $M'$. 
In particular, the two-form $f^*\omega'$ of $C^{r-1}$-regularity is $C^1$-close to $\widehat \omega$. 
Assume by contradiction that $\omega'$ is stabilizable, i.e. there is a (smooth) one-form $\lambda'\in \Omega^1(M')$ stabilizing $\omega'$. 
Then, $f^*\lambda'$ is a stabilizing one-form of $f^*\omega'$ of class $C^{r-1}$. 
Using that $r\geq 3$ we reach a contradiction with \Cref{lem:nonstab} together with Remark \ref{rem:reguAnosov}. 
This concludes the proof. 
\hfill \qedsymbol{}

\begin{Remark}
    Convex hypersurfaces in contact manifolds bear some analogies with stable hypersurfaces in symplectic manifolds, like the existence of a tubular neighborhood foliated by hypersurfaces with diffeomorphic characteristic foliations. A recent preprint by Chaidez \cite{Ch} shows that on any high dimensional contact manifold there are hypersurfaces (in certain isotopy classes) with a $C^2$-neighborhood that contains no convex hypersurface, thus establishing a version of Theorem \ref{thm:main1} in the contact setting.
\end{Remark}

\begin{Remark}\label{rem:regu}
    The argument breaks down if one looks at a hypersurface that is only $C^2$-close to $M$. In this case, the new invariant submanifold is given by an embedding $f:N\rightarrow M'$ that is only $C^1$-close to $e$. In particular, the Hamiltonian structure $f^*\omega'$ is only $C^0$-close to $\widehat \omega$ and we cannot apply Lemma \ref{lem:nonstab}. Indeed, the characteristic foliation is spanned by a vector field that is (a priori) only $C^0$-close to an Anosov vector field, which does not imply that it is orbit equivalent to that Anosov field.
\end{Remark}

\bigskip

\section{Non-density of non-degenerate stable Hamiltonian structures} \label{s:T2}

The goal of this section is to prove Theorem \ref{thm:stably_degenerate} about the non-density of non-degenerate stable Hamiltonian structures. First, we give a simple degeneracy criterion for a Hamiltonian structure: the existence of a closed orbit in the regular level set of a first integral of the characteristic foliation, see \Cref{lem:degenerate_Hamilt_Reeb}. 
Secondly, we make a semi-local construction of a \shs in the product of any contact manifold with $I\times S^1$, which satisfies that any \shs in a $C^2$-neighborhood of this one necessarily satisfies the degeneracy criterion, this is the content of \Cref{prop:perturbing_suspension}. 
Finally, we prove \Cref{thm:stably_degenerate} by showing that this local construction can be introduced via a stable homotopy in a class of stable homotopy classes that we call ``regular" and that we introduce below.

\subsection{Regular stable homotopy classes}\label{ss:regular}

We first give the precise definition of the stable homotopy classes to which our theorem applies and then find sufficient conditions for a \shs to be in this type of homotopy classes.
\begin{defi}\label{def:goodSHS}
    A stable homotopy class in $M$ is \emph{regular} if it is representable by a \shs $(\lambda,\omega)$ such that $\d\lambda$ is a non-zero constant multiple of $\omega$ on some open set $U\subset M$.
\end{defi}
Notice that up to scaling $\lambda$, we might assume that in some open set (namely, a connected component of $U$) we have $\d\lambda=\pm \omega$.
In three dimensions, it follows from \cite[Proposition 3.31]{CV} that every stable homotopy class is regular.
In arbitrary dimensions, a contact type two-form $d\alpha$ trivially belongs to a regular stable homotopy class. 
We now also prove that one sufficient condition for $(\lambda,\omega)$ to be in a regular stable homotopy class is that $\lambda$ is closed on some open set.

\begin{lemma}
    Let $(\lambda,\omega)$ be a \shs such that $d\lambda=0$ in some open set $U\subset M$. Then $(\lambda,\omega)$ belongs to a regular stable homotopy class.
\end{lemma}
\begin{proof}
    We will first show by a simple closing lemma argument that, up to stable homotopy, we can assume that the Reeb field admits a periodic orbit contained in the interior of the set $\mathcal{O}_{\lambda}=\{d\lambda=0\}$. Then, we will modify the \shs so that the first return map along this periodic orbit is the identity in a small enough transverse disk, and use this to deform $\lambda$ to be a contact form there.
    
    If there is no periodic orbit in $\mathcal{O}_{\lambda}$, we argue as follows. 
    The set $U$, which we assume to be connected, belongs to a connected component $V$ of $\mathcal{O}_{\lambda}$, and $V$ is invariant by the flow $\varphi_X^t$ of the Reeb field $X$ of $(\lambda,\omega)$.
    The interior of $V$ is also invariant, since if a point admits a small neighborhood where $d\lambda=0$, this will hold for any point in the orbit of $X$ through $p$ by continuity of the flow and the fact that $\d\lambda$ is itself invariant under the flow.
    
    Consider then a Hamiltonian flow-box neighborhood $B=[-1,1]\times D^{2n}$ contained in $U$, i.e. with coordinates $z,x_i,y_i$ such that $X$ is parallel to $\pp{}{z}$ and $\omega= \sum_{i=1}^n dx_i\wedge dy_i$. 
    Choose a Poincaré recurrent point $p=(0,q)\in B$, i.e.\ such that there is some time $\tau>0$ with $\varphi_\tau(p)=(0,q')$ for some point $q'$ different from $q$. 
    For later use, we define $U'$ to be an open set diffeomorphic to a solid torus, given by the union of a small neighborhood of $\{\varphi_t(q), t\in [0,\tau]\}$ and a small neighborhood of a segment $(0,q_s), s\in [0,1]$ inside $\{0\}\times D^{2n}\subset B$, with $q_0=q$ and $q_1=q'$.  
    As argued before, the set $\varphi[0,\tau]$ is contained in the interior of $V$, so we can assume that $U'$ is also in the interior of $V$.
    
    Equip $M\times (-\varepsilon,\varepsilon)$ with the two-form $\Omega= \omega + d(t\lambda)$, where $t$ is the coordinate in the second factor. This form is symplectic for $\varepsilon$ small enough. The Hamiltonian structure $\omega$ is obtained by restriction of $\Omega$ to $H^{-1}(0)$, where the Hamiltonian is $H=t$. 
    The Hamiltonian $C^1$-closing lemma\footnote{See the precise statement we use in \cite[Section 2.3 ``Closing lemma pour les Hamiltoniens dans les surfaces d'énergie"]{Ar}.} \cite{PR} implies that there exists a Hamiltonian function $\widetilde H$ that is arbitrarily $C^2$-close to $H$, such that $\widetilde H=H$ everywhere besides in $U'\times (-\delta,\delta)$ for an arbitrarily small $\delta<\varepsilon$, and also satisfying the following property: if $\iota:M \rightarrow M\times (-\varepsilon,\varepsilon)$ denotes the embedding such that $\iota(M)=\widetilde H^{-1}(0)$, the characteristic foliation of the Hamiltonian structure $\widetilde \omega= \iota^*\Omega$ admits a closed orbit in $U'$. 
    Notice that $\widetilde \omega$ is arbitrarily $C^1$-close to $\omega$ and equal to it away from $U'$, and hence $\lambda$ is also a stabilizing one-form of $\widetilde \omega$. 
    Furthermore $\widetilde \omega$ is cohomologous to $\omega$ and thus $(\lambda_r=\lambda, \omega_r=(1-r) \omega + r \widetilde \omega), r\in [0,1]$ defines a stable homotopy. 
    By a slight abuse of notation, we keep denoting by $\omega$ the new Hamiltonian structure $\widetilde \omega$.  
    \\

    This shows that we can assume that $(\lambda,\omega)$ admits a periodic orbit $\gamma$ contained in the interior of $\{p\in M \mid d\lambda|_p=0\}$. 
    The arguments in \cite[Proposition 47]{C}, which work in any dimension, then show that we can $C^1$-perturb $\omega$ to another cohomologous Hamiltonian structure $\widehat \omega$ such that $\widehat \omega=\omega$ away from a small neighborhood of a point $p\in \gamma$ and that the following property is satisfied: 
    in a small neighborhood of the origin of a small disk transverse to $\gamma$, the first return map of any vector field spanning $\ker \widehat \omega$ is the identity. 
    We choose $\widehat \omega$ sufficiently $C^1$-close to $\omega$ so that a linear interpolation between both forms is a homotopy of maximally non-degenerate two-forms.
    As $\lambda$ is closed near $p$, it is also a stabilizing one-form of $\widehat \omega$ and $(\lambda,\widehat \omega)$ is then stable homotopic to $(\lambda,\omega)$. 
    We can find coordinates $(\psi,x_i,y_i)$ on a neighborhood $V\cong S^1\times D^{2n}$ of $\gamma$ such that $\widehat \omega= \sum_{i=1}^n dx_i\wedge dy_i$ and $\lambda= dz + dg$ with $g\in C^\infty(D^{2n})$. 
    Finally, let $F:D^{2n}\longrightarrow \mathbb{R}$ be a cut-off function equal to $1$ near the origin and equal to $0$ near the boundary. Consider the stable homotopy $(\lambda_s, \widehat \omega)$ where
    $$\lambda_s= dz + d((1-sF)g) + sF x_idy_i.$$
    For $s=1$, the one-form $\lambda_1$ satisfies $d\lambda=\widehat \omega$ in a small neighborhood of $\gamma$. This proves that $(\lambda, \omega)$ belongs to a regular stable homotopy class. 
\end{proof}
A particular application of the previous lemma is that symplectic mapping tori, i.e. stable Hamiltonian structures such that $\lambda$ is closed, belong to a regular stable homotopy class. Products of contact and symplectic manifolds are also regular:
\begin{lemma}
    Let $\omega$ be a Hamiltonian structure on $M=N\times W$ with $\omega=d\alpha\oplus \widehat \omega$ where $d\alpha$ is a contact type two form in $N$ and $(W,\widehat \omega)$ is a symplectic manifold. 
    Then $\omega$ belongs to a regular stable homotopy class.
\end{lemma}
\begin{proof}
 We denote by $2k+1$ the dimension of $N$, and by $2n+1$ the dimension of $M$. 
 Consider the stable Hamiltonian structure $(\alpha, \omega)$. 
 By standard contact neighborhood theorems, one can choose a contact form $\alpha_1 \in C^\infty(N)$ defining $\ker \alpha$ such that there is a neighborhood $V\cong S^1\times D^{2k}\subset N$ of an embedded closed curve transverse to $\xi$ with coordinates $\psi, r_i, \theta_i$ where $\alpha_1= d\psi + \sum_{i=1}^k r_i^2d\theta_i $. Now, a homotopy of contact forms $\alpha_t=f_t\alpha$ such that $f_t$ is an everywhere positive function and $\alpha_1=f_1\alpha$ naturally induces a stable homotopy $(\alpha_t,\omega_t)=(\alpha_t, d\alpha_t \oplus \omega)$. 
 If we choose a small disk $U\cong D^{2(n-k)}\subset W$ with Darboux coordinates $\rho_i, \varphi_i$ for $\widehat \omega$, we have that in $V\times U$ the forms write as $\alpha_1= d\psi + \sum_{i=1}^k r_i^2d\theta_i $ and $\omega_1= \sum_{i=1}^k r_idr_i\wedge d\theta_i + \sum_{i=1}^{n-k} \rho_i d\rho_i \wedge d\varphi_i$. 
 Let $F:D^{2k}\times D^{2(n-k)}\to \R$ be a bump function such that $F=1$ in a small neighborhood $U'$ of the origin, and $F=0$ away from a slightly larger neighborhood of the origin. 
 The one-form 
 $$\alpha_2= \alpha_1 + \sum_{i=1}^{n-k} F\rho_i^2 d\varphi_i $$
 is a stabilizing one-form of $\omega_1$ satisfying $d\alpha_2=\omega_1$ in $S^1\times U'$. 
 Since the space of stabilizing one-forms for a given Hamiltonian structure is path-connected, we deduce that $(\alpha,\omega)$ is stable homotopic to the \shs $(\alpha_2,\omega_1)$ with an open set where $d\alpha_2=\omega_1$.
 This proves that $(\alpha,\omega)$ belongs to a regular stable homotopy class, as desired.
\end{proof}

\subsection{A degeneracy criterion for Hamiltonian structures}
\label{sec:degeneracy_criterion}

We give here a simple criterion for a Hamiltonian structure to be degenerate, in terms of the existence of a periodic orbit contained in the regular level set of a first integral of its characteristic foliation.

\begin{lemma}\label{lem:degenerate_Hamilt_Reeb}
    Let $\omega$ be a Hamiltonian structure and $X$ a vector field spanning $\ker \omega$. Let $L$ be a connected component of a regular level set of any first integral $f$ of $X$. If $X$ has a periodic orbit in $L$, then $X$ is degenerate.
\end{lemma}

\begin{proof}
    Denote by $N=f^{-1}(c)\subset M$ the regular level set of the first integral $f$.
Because values near to $c$ will also be regular for $f$, there is a neighborhood $N\times (-\epsilon,\epsilon)$ of $N\simeq N\times\{0\}$ inside $M$ such that $X$ is tangent to the level sets $N\times\{t\}$.
Let now $\lambda$ be a $1$-form such that $\lambda(X)=1$, so that $\mu\coloneqq\lambda\wedge\omega^n$ is a volume form on $M$.
Note also that by construction $\mu$ is automatically preserved by the flow of $X$.
As $X$ is tangent to all level sets $N\times\{t\}$ near $N=N\times\{0\}$, if we write near $N$ the volume form $\mu$ as $\mu=dt\wedge \mu_{N,t}$, the volume form $\mu_N=\mu_{N,0}$ is preserved by $X\vert_N$.

Note also that the restriction $\omega_N$ to $N$ has a kernel of constant rank $2$, which is more precisely spanned by $X$ and another non-zero vector field which we denote by $Y$.

Denote by $\gamma$ the closed orbit of $X$ in $N=N\times\{0\}$.
Consider a point $p\in\gamma$, and let $\Phi\colon \xi_p \to \xi_p$ the linearized first return map of $X$ at $p$, where $\xi = TN/\langle X\rangle$ over $N$.
Note then that $\xi$ has a natural splitting $\eta \oplus \langle Y\rangle $, where $\eta$ is any (arbitrary) complementary to $\langle Y \rangle$ inside $\xi$. Moreover, because $X$ preserves $\omega$ and $N$, it also preserves $\omega_N$ and hence $Y$.
In particular, $(\eta,\omega_N)$ is naturally a symplectic vector bundle over $N$. The restriction of the linearized first return map, seen as map $\Phi\colon \eta_p\oplus\langle Y(p)\rangle \to \eta_p\oplus\langle Y(p)\rangle$ splits into two blocks, one given by the invariant subspace $\langle Y \rangle$, and the other block given by $\eta$ satisfies that $\mathrm{pr}_{\eta_p}\circ\Phi\vert_{\eta_p}$ is a symplectic matrix that in particular has determinant $1$.
The fact that $\Phi$ also globally has determinant $1$, because the flow of $X$ preserves the volume form $\mu_N$ on $N$, then immediately implies that necessarily $\Phi(Y(p))=Y(p)$.
In other words, the linearized first return map along $\gamma$ has $1$ as eigenvalue, and $X$ is hence degenerate, as we wanted to show.
\end{proof}

We now proceed to construct stable Hamiltonian structures that satisfy this criterion even after perturbation.

\medskip

\subsection{Local construction of transversely hyperbolic invariant tori}
\label{sec:transversly_hyperbolic_orbits}

We want to apply Lemma \ref{lem:degenerate_Hamilt_Reeb} to every \shs close to a given one. This would imply in particular that the presence of degenerate periodic orbits holds in some open set of stable Hamiltonian structures. 
Degenerate periodic orbits do not persist under perturbations of a vector field, however, notice that we are perturbing the pair $(\lambda,\omega)$ and not the Reeb vector field directly. 
The next proposition shows that there exists a particular construction of a \shs that always admits a periodic orbit in a level set of one of the integrals introduced in Lemma \ref{lem:first_integrals}. 
The idea is to construct a \shs that has, for $s\in (-\delta,\delta)$, an $s$-parametric family of invariant tori $T_s$, each contained in the regular level sets $f^{-1}(s)$ of one of the integrals $f$, and these tori are moreover normally hyperbolic invariant submanifolds of the Reeb flow restricted to $f^{-1}(s)$. We do it in a way that the ``rotation vector" of the flow along $T_s$ varies with $s$; we arrange this property as this will later be used to deduce that any \shs $(\lambda',\omega')$ close to this one will necessarily have a closed orbit in a level set of the first integral $f'$ (a perturbation of $f$) of its Reeb field. 

For the next statement, let us recall that a closed orbit of a vector field is called \emph{hyperbolic} if the linearized first return map has no eigenvalue with absolute value equal to $1$. 

\begin{prop}
    \label{prop:perturbing_suspension}
    Let $N$ be a closed $(2n-1)$-dimensional manifold, with $n\geq 2$, and $\alpha$ be a contact form on it for which there exists a closed hyperbolic Reeb orbit $\gamma \subset N$. 
    Consider the \shs $\left(\lambda,\omega\right)=\left(\d\psi ,\d\left(e^t\alpha\right)\right)$ on $M=(-1,1)_t\times N \times S^1$, where $\psi$ is an angle coordinate in the $S^1$ factor.
    Then, there are $\delta>0$ and a \shs $(\lambda',\omega')$ on $(-1,1)_t\times N \times S^1$ satisfying the following properties:
    \begin{enumerate}
        \item\label{item:perturbing_suspension_equality_near_boundary} $(\lambda',\omega')=(\lambda,\omega)$ on 
        $\left(\left(-1,-1+\delta\right)\cup\left(1-\delta,1 \right) \right)
        \times N \times S^1$;
        \item\label{item:perturbing_suspension_homotopy} $(\lambda',\omega')$ is stable homotopic to $(\lambda,\omega)$ relative to 
        $\left(\left(-1,-1+\delta\right)\cup\left(1-\delta,1\right)\right) \times N
        \times S^1$;
        \item\label{item:perturbing_suspension_persisting_S1_fam_closed_orbits} if $(\widehat \lambda,\widehat \omega)$ is a stable Hamiltonian structure that is sufficiently $C^2$-close to $(\lambda',\omega')$, then its Reeb vector field $\widehat R$ admits a periodic orbit contained in the regular level set of the first integral $$g= \frac{d\widehat \lambda\wedge \widehat \omega^{n-1}}{\widehat \omega^n}$$ of $\widehat R$ contained in the interior of $M$.
    \end{enumerate}
\end{prop}

\begin{proof}
    We look at $(-1,1)_t\times N \times S^1_\psi$ as the quotient of $(-1,1)_t\times N\times \R_\psi$ under the $\Z$-action $\rho$ generated by $1\cdot (t,q,\psi) = (t,q,\psi-1)$.
    
    Let $\tau$ be the period of the periodic orbit $\gamma$. Choose some small $\delta>0$ and let $\chi\colon (-1,1)\to [0,\tau]$ be a cut-off function that is equal to $0$ on $(-1,-1+\delta]$ and on $[1-\delta,1)$, equal to $\tau$ at $0$, strictly increasing on $(-1+\delta,0)$, and strictly decreasing on $(0,1-\delta)$.
    Denote also by $\phi_{R_\alpha}^s$ the flow of the Reeb vector field $R_\alpha$ of $\alpha$ on $N$ at time $s$.
    We then consider the quotient $X$ of $(-1,1)_t\times N \times \R_\psi$ by the (new) $\Z$-action $\rho'$ generated by $1\cdot (t,q,\psi)=(t,\phi_{R_\alpha}^{\chi(t)}(q),\psi-1)$.

    There is a homotopy of actions $(\rho_s)_{s\in[0,1]}$ from $\rho$ to $\rho'$ relative to the region $((-1,-1+\delta]\cup[1-\delta,1))\times N\times \R$:
    this just follows from the fact that one can find a $C^\infty$-small linear (in $s$) interpolation $(\chi_s)_{s\in[0,1]}$ of cut-off functions between $\chi_0=0$ and $\chi_1=\chi$.
    Then, by looking at the fiber bundle 
    \[
    \cup_{s\in[0,1]} X_s \to [0,1] \, ,
    \text{ where } \,
    X_s = ((-1,1)\times N \times \R)/\rho_s \, ,
    \]
    we can find a $[0,1]$-family of diffeomorphisms $\Phi_s\colon X_s\xrightarrow{\sim} (-1,1)\times N \times S^1$ which starts at the identity $\Phi_0=\Id\colon X_0 \xrightarrow{=}(-1,1)\times N \times S^1$.
    This can be done for instance by considering the parallel transport (starting at the fiber over $s=0$) of any auxiliary connection on $\cup_s X_s\to [0,1]$.
    For later use, we also note that said parallel transport from the fiber over a certain $s_0$ to that over some other $s_1$ sends the slice $\{t=t_0\}$ over $s=s_0$ to $\{t=t_0\}$ over $s=s_1$. 
    This is true because $\cup_s X_s$ is obtained also as the quotient of $[0,1]_s\times(-1,1)_t\times N \times \R$ by the action $\varrho$ given by $1\cdot(s,t,q,\psi)=(s,\rho_s(t,q,\psi))$.

    \medskip
    
    We now consider in $(-1,1)\times N \times \R_\psi$ a pair of the form 
    \begin{equation}\label{eq:firstSHS}
        (\widetilde \lambda, \widetilde \omega)= \left(\d\psi + h(t)\alpha , \d\left(e^t\alpha\right)\right),
    \end{equation}
    where $h$ is a non-zero compactly supported function in $(-1,1)$. 
    First, note that this is a stable Hamiltonian structure: 
    indeed, $\d\widetilde{\lambda} = {h}'(t)\d t\wedge \alpha + {h}(t)\d \alpha$, so that $\ker(\widetilde\omega)=\langle \partial_\psi\rangle\subset \ker(\d\widetilde\lambda)$ and 
    \[
    \widetilde\lambda\wedge \widetilde\omega^{n} = n e^{nt} \d\psi \wedge \d t \wedge \alpha \wedge \d\alpha^{n-1} >0 \, .
    \]
    We also note that 
    \begin{equation}
    \label{eqn:lemma_perturbing_suspension_reg_value}
    \begin{split}
    \d\widetilde\lambda\wedge \widetilde\omega^{n-1} = & \, 
    e^{(n-1)t}
    ({h}' \d t \wedge \alpha + {h} \d\alpha)\wedge 
    (\d t \wedge \alpha + \d\alpha)^{n-1} \\
    = &\,
    e^{(n-1)t} [{h}'+(n-1){h}] \d t \wedge\alpha \wedge \d\alpha^{n-1} \\
    = & \,
    \frac{1}{n}e^t[{h}' + (n-1){h})] \,\widetilde\omega^n
    \,.
    \end{split}
    \end{equation}
    For later use, we choose ${h}$ so that $e^t[{h}'+(n-1){h}]$ has $0$ as a regular value.
    Notice that $(\widetilde{\lambda} , \widetilde{\omega} )$ is in fact $\rho_s$-invariant for every $s\in[0,1]$.
    In particular, we get well-defined stable Hamiltonian structures $(\widehat\lambda_s,\widehat\omega_s)$ on $X_s$, for every $s\in [0,1]$. 
    In fact, $(\widehat\lambda_0,\widehat\omega_0)$ is nothing else than the pair of forms induced by $(\widetilde \lambda, \widetilde \omega)$ on $X_0=(-1,1)\times N \times S^1$. 
    This induces a one-parameter family of stable Hamiltonian structures $(\lambda_s,\omega_s)$ on $(-1,1)\times N\times S^1$, given by the pullbacks $(\Phi_s^{-1})^*(\widehat\lambda_s,\widehat\omega_s)$.
    We claim that $(\lambda', \omega') = (\lambda_1, \omega_1)$ satisfies the conclusion of the theorem.

    \medskip
    
    First, by construction $(\lambda_s,\omega_s)$ coincides with $(\lambda,\omega)$ on $((-1,-1+\delta)\cup(1-\delta,1))\times N \times S^1$. 
    In particular, \Cref{item:perturbing_suspension_equality_near_boundary} is satisfied.
    
    Secondly, the pair $(\lambda,\omega)$ is clearly homotopic to $(\widetilde \lambda_0, \widetilde \omega_0)$ (simply consider a homotopy of functions $h_s(t)$ such that $h_1=h$ and $h_0\equiv 0$ and pairs of the form \eqref{eq:firstSHS}), thus $(\lambda,\omega)$ and $(\lambda',\omega')$ are homotopic. Thus \Cref{item:perturbing_suspension_homotopy} also holds.

    \medskip
    
    It is then only left to argue that \Cref{item:perturbing_suspension_persisting_S1_fam_closed_orbits} holds.
    Let $Z$ be the Reeb field of $(\lambda',\omega')$. 
    It is tangent to the hypersurfaces $\{t=c\}$, for $c\in (-1+\delta, 1-\delta)$, and we claim that there is a family of invariant tori $T_c\subset \{t=c\}$ of $X$ which are normally hyperbolic in $\{t=c\}$, for $c$ sufficiently close to $0$. 
    Indeed, notice that the Reeb vector field of $(\widetilde\lambda,\widetilde\omega)$ on $(-1,1)\times N \times \R$ is tangent to the slices $\{t=\const\}$ and, after taking the quotient by $\rho$ and seeing it as a vector field on $X_1$, it is still a suspension flow, with global cross-section diffeomorphic to $(-1,1)\times N$ (obtained from the quotient of $\{\psi=0\}\subset (-1,1)\times N \times \R$ under the action $\rho_1$). The first return map of the Reeb field is conjugate to
    \begin{align*}
        \psi \colon (-1,1)\times N &\longrightarrow (-1,1)\times N\\
    (t,q) &\longmapsto \left(t,\phi_{R_\alpha}^{\chi(t)}(q)\right).
    \end{align*}
    We claim that the closed hyperbolic orbit $\gamma$ of $R_\alpha$ (from the statement) gives rise to an invariant torus $T_c$ of $Z$ in each $\{t=c\} \subset X_1$. 
    Indeed, the first return map $\psi$ preserves the circle $S_c=\{c\}\times \gamma \subset (-1,1)\times N$ on each slice $\{t=c\}$ since $\gamma$ is a closed orbit of the Reeb field of $\alpha$. 
    The vector field $Z$, which is conjugate to the suspension of $\psi$, has then for each circle $S_c$ a corresponding invariant torus $T_c=\{c\}\times\gamma \times S^1\subset X_1$. 
    The foliation induced on this torus is rational, respectively irrational when $\chi(t)$ is a rational, respectively irrational, number. 
    Notice that the pullback of $\omega'$ to $\{t=c\}$ has a kernel generated by $Z$ and by the Reeb field of the contact form $\alpha$ on each $\{c\}\times N \times \{\psi\}$. 
    In particular, the normal directions of the torus $T_c$ inside $\{t=c\}$ coincide with the normal directions of $\{c\}\times \gamma\times \{\psi\}$ inside $\{c\}\times N \times \{\psi\}$ for each $\psi\in S^1$. 
    The hyperbolicity of $\gamma$ together with the fact that the linearized Poincar\'e map of the Reeb flow is symplectic ensures that at any point $q\in \gamma$, the differential $d\phi_{R_{\alpha}^\tau}$ has three eigenvalues, namely $1$ (in the $R_{\alpha}$ direction), $\lambda$ and $\frac{1}{\lambda}$ for some $\lambda>1$.
    Notice that $\phi|_{t=0}$ is conjugate to $\phi_{R_\alpha}^\tau$, which implies that $S_0$ is a normally hyperbolic invariant submanifold of $\phi|_{\{t=0\}}$. Since $\chi(t)$ is close to the period $\tau$ of $\gamma$ for $t\in (-\delta',\delta')$, the diffeomorphism $\phi|_{\{t=c\}}$ is close to $\phi|_{\{t=0\}}$ for any $c\in (-\delta',\delta')$, and thus the circles $S_c$ are normally hyperbolic invariant submanifolds of $\psi|_{\{t=c\}}$ for $\delta'$ small enough. 
    Hence each torus $T_c$ with $ c\in (-\delta',\delta')$ is a normally hyperbolic invariant submanifold of $X$ inside $\{t=c\}$. Notice as well that $\chi(t)$ is non-constant near $t=0$, and thus the first return map along $S_c$ is a rotation of non-constant angle for $c\in (-\delta',\delta')$. In particular, the rotation number of $\psi|_{S_c}$ is non-constant for $c\in (-\delta',\delta')$.\\

    Consider now $(\widehat \lambda,\widehat \omega)$ a stable Hamiltonian structure which is $C^2$-close to $(\lambda',\omega')$.
    Then $\d \widehat \lambda \wedge \widehat \omega^{n-1}= g \cdot \widehat \omega^n$, for a function $g$ that is $C^1$-close to the function 
    $$f:=\frac{1}{n}e^t[h'+(n-1)h]$$
    in Equation \eqref{eqn:lemma_perturbing_suspension_reg_value}, on the quotient $(-1,1)\times N\times S^1$ of $(-1,1)\times N \times \R$ via $\rho_1$. 
    In particular, the regular level sets $g^{-1}(t)$ with $t\in (-\delta',\delta')$ are $C^1$-close to the regular level sets $f^{-1}(t)$ of $f$ with $t\in (-\delta',\delta')$. Each connected component of any of these level sets is of the form $\{c\}\times N \times S^1$ for some $c\in (-\delta', \delta')$. Thus, up to a $C^1$-small isotopy, we can assume that each connected component of $g^{-1}(t)$ with $t\in(-\delta',\delta')$ is of the form $\{c\}\times N \times S^1$.
    According to \Cref{lem:first_integrals}, the Reeb field $\widehat R$ of $(\widehat \lambda,  \widehat \omega)$ is then tangent to these level sets,  and on each $\{t=c\}$ with $c\in (-\delta',\delta')$ the dynamics is $C^1$-close to the dynamics of the Reeb vector field $R'$ of $(\lambda',\omega')$ along $\{t=c\}$.
    Thus $\widehat R$ still admits a global cross-section $(-\delta',\delta')\times N\times \{0\}$ where the first return map $\widehat F$ is $C^1$-close to $\psi$.
    Notice then that the first return maps $\widehat F_c=\widehat F|_{\{t=c\}}$ is a one-parametric family of diffeomorphisms of $N$ which are $C^1$-close to $F_c:=\phi_{R_\alpha}^{\chi(c)}$. 
    The parametric family of diffeomorphisms $F_c$ admits a parametric family of normally hyperbolic invariant circles for $c\in (-\delta',\delta')$, namely $S_c=T_c\cap\{\psi=0\}$, and hence by the (parametric) persistence of normally hyperbolic invariant submanifolds (\Cref{thm:normal_hyperbolic_stable} and \Cref{rem:NHSparam}) there is a parametric family of circles $\widehat S_c$, in general of $C^1$-regularity, invariant by $\widehat F_c$ for $c\in (-\delta',\delta')$. 
    For the flow of $\widehat R$, this implies the existence of a family of normally hyperbolic invariant tori.
    
    Since the rotation number of $\psi|_{S_c}$ is non-constant for $c\in (-\delta',\delta')$ and the dynamics of $\widehat F_c$ along $\widehat{S}_c$ is $C^0$-close to $\psi|_{S_c}$, \Cref{prop:rotcont} implies that the rotation number of $\widehat F|_{\widehat S_c}$ for $c\in (-\delta',\delta')$ is also non-constant. 
    The dynamics of $\widehat F$ along each $\widehat S_c$ might no longer be conjugate to a rotation, as it was in the case of $F$ along $S_c$ for the unperturbed flow; however, by the fact that the rotation number of $\widehat F_c$ along $\widehat S_c$ is not constant, there must be (infinitely many) values of $c$ for which this number is rational.
    By \Cref{prop:rotperiodic}, for any such value $c$ there must be a periodic point of $\widehat F_c$ in $\widehat S_c$. 
    These periodic points give rise to periodic orbits of $\widehat R$.
    This concludes the proof of \Cref{item:perturbing_suspension_persisting_S1_fam_closed_orbits}.
 \end{proof}
 
Observe that any stable Hamiltonian structure $C^2$-close to $(\lambda',\omega')$ in the statement of Proposition \ref{prop:perturbing_suspension} is degenerate by Lemma \ref{lem:degenerate_Hamilt_Reeb}.
We also underline the fact that the proof above relies in a fundamental way on \Cref{lem:first_integrals}, i.e.\ the existence of natural first integrals of the Reeb field of stable Hamiltonian structures. Further consequences of the existence of these first integrals will be explored in future work.

\medskip

\subsection{Proof of \Cref{thm:stably_degenerate}}
\label{sec:proof_stably_degenerate}
To prove the statement, we will insert the local construction obtained in Proposition \ref{prop:perturbing_suspension} via a stable homotopy.\\

Let $(\lambda,\omega)$ be a stable Hamiltonian structure in $M$; since we are in a regular stable homotopy class by assumption, we can arrange that there is a connected open subset $U\subset M$ where $\d\lambda=\pm\omega$. 
We assume here $\d\lambda = \omega$ on $U$ for simplicity, as the proof in the case where $\d\lambda=-\omega$ is completely analogous. 

Pick a closed curve $\gamma\subset U$ transverse to the contact structure $\xi=\ker\lambda|_U$ defined on $U$.
By standard contact neighborhood theorems, one can find a contact homotopy $(\lambda_s)_{s\in[0,1]}$, starting at $\lambda_0=\lambda$, supported in a neighborhood of $\gamma$, and such that in a smaller neighborhood $S^1_\psi\times D^{2n}$ of $\gamma\simeq S^1_\psi\times\{0\}$ we have $\lambda_1=\lambda' = \d\psi + \sum_{i=1}^n r_i^2 \d \theta_i$, where $\psi\in S^1$ and $(r_i,\theta_i)$ are polar coordinates on each of the $D^2$ factors of $D^{2n}=D^2\times \ldots\times D^2$.
The resulting homotopy $(\lambda_s,\omega_s)$, where  $\omega_s=\d\lambda_s$, extends as $\lambda_s=\lambda$ and $\omega_s=\omega$ away from a neighborhood of $\gamma$, and is a stable homotopy. 
 
We consider now another homotopy $(\lambda_s,\omega_s)_{s\in[1,2]}$ given by $\omega_s=\omega_1$ and 
 \[
 \lambda_s = \d\psi + \rho_s(r)\sum_{i=1}^n r_i^2\d\theta_i 
 \]
 for all $s\in[1,2]$,
 where $r$ is the radial coordinate in $D^{2n}$ and $\rho_s$ is a family of functions all equal to $1$ near $1$, such that $\rho_1\equiv 1$ and such that $\rho_2$ is equal to $0$ near $0$.
 For a small enough radius $\delta$, $(\lambda_2,\omega_2)$ just restricts to $(d\psi, 2\sum_{i=1}^n r_i dr_i \wedge d\theta_i)$ on $S^1\times D^{2n}_\delta$, and hence its Reeb flow is degenerate, as each orbit is of the form $S^1\times\{pt\}$ in such region.
 
 Now, by the Weinstein neighborhood theorem for Lagrangian submanifolds, the torus
 $T^n=S_\epsilon^1\times \ldots S_\epsilon^1 \subset D^2\times \ldots D^2 = D^{2n}$ 
 where $S^1_\epsilon$ is the circle of radius $\epsilon$ in each factor $D^2$, 
 admits an open neighborhood $S^1\times D_\delta T^*T^n\subset S^1\times D^{2n}$ on which $(\lambda_2=\d\psi, \omega_2 = \d\lambda_{std})$, 
 where $\lambda_{std}$ is the standard Liouville form on $D_\delta T^*T^n$ and $D_\delta T^*T^n$ is the subset of $T^*T^n$ made of those covectors of 
 norm (w.r.t.\ the standard flat metric for instance) at most a certain small enough $\delta>0$.

 Choose a metric $g$ on $T^n$ such that its geodesic flow has a closed orbit which is hyperbolic, such a metric exists e.g. by \cite{KT}.
 Up to rescaling, we can assume that the $g$-unit cosphere bundle $S_gT^*T^n$ is contained in $D_\delta T^*T^n$; denote by $N$ such cosphere bundle, and by $\alpha$ the naturally induced contact form on it, whose Reeb flow is exactly the geodesic flow of $g$.
 Then, for $\eta>0$ small enough, there is a neighborhood $V=[-\eta,\eta]\times S^1_\psi\times N$ of $S^1 \times N = S^1\times S_gT^*T^n\subset S^1\times D^{2n}$ where
 \[
 \lambda_2 = \d\psi \, ,
 \quad
 \omega_2 = \d(e^t\alpha) \, .
 \]

 Now, applying \Cref{prop:perturbing_suspension} to the restriction of the stable Hamiltonian structure $(\lambda_2,\omega_2)$ on $V$ gives then a further stable homotopy $(\lambda_s,\omega_s)_{s\in[2,3]}$ among stable Hamiltonian structures, relative to a neighborhood of the boundary of $V$; by extending trivially outside of $V$, this can in particular naturally be seen as a homotopy (still denoted the same) of the ambient stable Hamiltonian structure. 
 The \shs $(\lambda'\coloneqq \lambda_3,\omega'\coloneqq \omega_3)$ is then the desired stable Hamiltonian structure on $M$. 
 Indeed, by \Cref{item:perturbing_suspension_persisting_S1_fam_closed_orbits} in Proposition \ref{prop:perturbing_suspension} and Lemma \ref{lem:degenerate_Hamilt_Reeb}, any $C^2$-close \shs is necessarily degenerate. This concludes the proof of \Cref{thm:stably_degenerate}. 
 \hfill \qedsymbol \\

We finish with a remark about approximations of stable hypersurfaces by non-degenerate stable hypersurfaces.

\begin{Remark}\label{rem:approx}
    Notice that Theorem \ref{thm:stably_degenerate} does not imply that the Hamiltonian structure $\omega'$ cannot be $C^2$-approximated (or even $C^\infty$) by a sequence $\omega_n$ of non-degenerate stabilizable Hamiltonian structures. Indeed, it might be the case that the $\omega_n$ do approximate $\omega'$ but that there exists no sequence of stabilizing one-forms $\lambda_n$ of $\omega_n$ that $C^2$-converge to $\lambda'$. An example of a similar phenomenon for a smooth family of stabilizing one-forms appears in \cite[Proposition 5.5]{CV}. To state this in terms of hypersurfaces, consider the symplectization $M\times (-\varepsilon,\varepsilon)$ of $(\lambda',\omega')$, i.e. we take the symplectic form $\omega'+d(t\lambda)$ where $t$ is a coordinate in the second factor. We cannot conclude that the stable hypersurface $M\times \{0\}$ cannot be $C^2$ or $C^\infty$-approximated by non-degenerate stable hypersurfaces.
\end{Remark}

\bibliographystyle{abbrv}
\bibliography{biblio}

\end{document}